\documentclass[10pt,a4paper]{article}
\usepackage{tikz}
\usepackage{float}
\usepackage{amsthm}
\usepackage{amsfonts}
\usepackage{amsmath}
\usepackage{amssymb}
\usepackage{stmaryrd}
\newtheorem{theorem}{Theorem}[section]
\newtheorem{corollary}[theorem]{Corollary}
\newtheorem{lemma}[theorem]{Lemma}
\newtheorem{proposition}[theorem]{Proposition}
\newtheorem{observation}[theorem]{Observation}
\theoremstyle{definition}
\newtheorem{definition}[theorem]{Definition}
\newtheorem{example}[theorem]{Example}
\newtheorem{remark}[theorem]{Remark}

\newcommand{\FR}[1]{\mathrm{FR}(#1)}
\newcommand{\Aut}[1]{\mathrm{Aut}(#1)}
\newcommand{\Out}[1]{\mathrm{Out}(#1)}
\newcommand{\Sym}[1]{\mathrm{Sym}(#1)}
\newcommand{\Inn}[1]{\mathrm{Inn}(#1)}

\author{Nils Leder\thanks{Supported by SFB 878 Groups, Geometry \& Actions.}}
\title{Serre's property FA for automorphism groups of free products}
\date{15.10.2018}
\begin{document}
\maketitle

\begin{abstract}
\noindent
We study the automorphism group $\Aut{G}$ of a free product $G$ of finite cyclic groups. We investigate the question in which cases $\Aut{G}$ has Serre's property FA.\\
In the case of two or three free factors, we prove that $\Aut{G}$ does not have property FA.  
However, if each free factor of $G$ occurs at least four times we show that $\Aut{G}$ does have property FA.
\end{abstract}


\section{Introduction}
\normalsize
In geometric group theory, it is a fruitful approach to study a group $G$ by considering its actions on geometric objects with nice properties. In this article, we investigate group actions on trees and the question whether a group $G$ has Serre's property FA, i.e. whether every action of $G$ on a simplicial tree has a global fixed point.

Although this notion is defined in geometric terms, it has strong algebraic consequences. By a theorem of Serre (see \S 6, Theorem 15 in \cite{S77}), a finitely generated group $G$ has property FA if and only if $G$ is not an amalgamated product and does not have a quotient isomorphic to $\mathbb{Z}$.

In particular, each finite group has property FA and therefore property FA only is an interesting feature in the case of infinite groups.

Considering the strong connection between free groups and their actions on trees, it is a very natural question to ask whether the automorphism group $\Aut{F_m}$ of a free group of rank $m$ has Serre's property FA. 
In \cite{B87}, Bogopolski showed that $\Aut{F_2}$ does not have Serre's property FA, but that $\Aut{F_m}$ satisfies property FA for all ranks $m \geq 3$.

In this paper, we study an analogous question for the case of a free product of finite cyclic groups, i.e. replacing each free factor $\mathbb{Z}$ of a free group by a finite cyclic group. Thus, we pick $m \geq 2, n_1, \ldots, n_m \geq 2$ and let $G= \ast_{i=1}^m \mathbb{Z}/n_i\mathbb{Z}$ be the free product of $m$ finite cyclic groups of order $n_i$ respectively. 
\begin{center}
\textbf{Main Question}: Does $\Aut{G}$ have Serre's property FA?
\end{center}
We first restrict to the \emph{pure} case where all the orders $n_i$ of the cyclic free factors agree.
Then, there also exists a sharp rank bound above which $\Aut{G}$ has property FA. In this case, the bound is $m=4$. Thus, we prove the following:

\begin{theorem}
Let $m,n  \geq 2$ and $G=\ast_{i=1}^m\mathbb{Z}/n\mathbb{Z}$ be the free product of $m$ copies of $\mathbb{Z}/n\mathbb{Z}$. Then, $\Aut{G}$ has property FA if and only if $m \geq 4$.
\end{theorem}

For $n=2$ the free product $G=\ast_{i=1}^m \mathbb{Z}/2\mathbb{Z}$ is also known as the universal Coxeter group of rank $m$. In this special situation, the result has been proven by Varghese in \cite{V18} (cf. Corollary B therein).

From the proof of Theorem 1.1 we obtain a decomposition of $\Aut{G}$ resp. $\Out{G}$ as an amalgamated product of finite groups in the cases $m =2$ resp. $m=3$.

We use Theorem 1.1 and a certain reduction technique to obtain the following two results for the \emph{mixed} case of free factors with different orders.

\begin{theorem}
Let $G=\ast_{i=1}^m \mathbb{Z}/n_i\mathbb{Z}$ be a free product of at least two non-trivial finite cyclic groups. If either a free factor $\mathbb{Z}/n\mathbb{Z}$ occurs exactly two or three times or if two different free factors $\mathbb{Z}/n_1\mathbb{Z}$ and $\mathbb{Z}/n_2\mathbb{Z}$ occur exactly once, then $\Aut{G}$ does not have FA.
\end{theorem}

\begin{theorem}
Let $G=\ast_{i=1}^m \mathbb{Z}/n_i\mathbb{Z}$ be a free product of finite cyclic groups. If each free factor $\mathbb{Z}/n\mathbb{Z}$ occurs at least four times, then $\Aut{G}$ has FA.
\end{theorem}

Note that the number of occurencies of a given finite cyclic group in the decomposition of $G$ as a free product of finite cyclic groups is an invariant of the group $G$. A proof of this fact is given later in Lemma 6.2.

The Theorems 1.2 and 1.3 treat many cases of mixed free products. The only remaining open case is when one free factor occurs once and all other free factors occur at least four times, e.g. $G \cong \mathbb{Z}/2\mathbb{Z}^{\ast 4} \ast \mathbb{Z}/3\mathbb{Z}$.\\ 
\newline
The paper is organised in five sections. In Section 2, we present the preliminaries on property FA, automorphisms of free products of finite cyclic groups, characteristic subgroups and the extension of actions to semi-direct products.

The absence of property FA in the case $m=2$ of Theorem 1.1 is proven in Section 3 by extending the (fixed point-free) action of $\Inn{G} \cong G$ on its Bass-Serre tree to the whole automorphism group $\Aut{G}$.

For $m=3$ in Theorem 1.1, we investigate in Section 4 the structure of $\Aut{G}$ in that case and then construct a fixed-point free action of $\mathrm{Out}(G)$ on a simplicial tree.

In Section 5, we prove the positive statement of Theorem 1.1 for $m \geq 4$. The proof requires the so called Subtree Cycle Lemma which allows us to prove that the Fouxe-Rabinovitch subgroup $\FR{G}$ of $\Aut{G}$ has a global fixed point whenever $\Aut{G}$ acts on a simplicial tree. This result immediately implies that the whole automorphism group $\Aut{G}$ has property FA.

The Section 6 finally treats the mixed case of free products of finite cyclic groups of different orders and includes the proofs of Theorems 1.2 and 1.3.

The proof of Theorem 1.2 makes use of Theorem 1.1 and the reduction by characteristic subgroups. The proof of Theorem 1.3 generalises that one of {Theorem 1.1} in the case $m \geq 4$.

\hfill \\
This article is part of my PhD thesis. I like to thank Olga Varghese for inspiring discussions and many helpful comments on this paper.

\section{Preliminaries}

We start by introducing Serre's property FA and providing some basic results. For further reading we recommend the monograph \emph{Trees} by Serre (\cite{S77}).

\subsection{Property FA}
The following definition is due to Serre (cf. \S 6.1, p.58 in \cite{S77}). The notation FA is an abbreviation for the french expression "fixe arbre" which means "fixing a tree". (This terminology is motivated by the fact that given a group action on a tree, the fixed point set is either empty or a subtree.)

\begin{definition} \hfill
\begin{itemize}
\item[$i)$] A group $G$ acts \emph{without inversion} on a tree $T$ if whenever $g \in G$ stabilises an edge $e \in E(T)$, $g$ fixes both endpoints of $e$ (in other words: $g$ fixes $e$ pointwise).
\item[$ii)$] A group $G$ is said to have \emph{property FA} if for every simplicial action without inversion of $G$ on a tree $T$ there is a global fixed point, i.e. there is a vertex $x \in V(T)$ such that $g(x)=x$ for all $g \in G$.
\end{itemize}
\end{definition}

\begin{example} \hfill
\begin{itemize}
\item[a)] Any finite group has property FA. More general, any finitely generated torsion group has property FA (cf. Example 6.3.1 in \cite{S77}).
\item[b)] For $m \geq 3$ the linear group $\mathrm{GL}_{m}(\mathbb{Z})$ and the automorphism group $\Aut{F_m}$ have property FA (cf. \cite{V02} for linear groups and \cite{B87} for free groups).
\item[c)] The group of integers $\mathbb{Z}$ acts by translations on the two-sided infinite line graph. This action is fixed point-free and hence $\mathbb{Z}$ does not have property FA (cf. proof of Theorem 15, p.58 in \cite{S77}).
\item[d)] A non-trivial amalgamated product $G=A \ast_C B$ does not have FA since $G$ acts without a global fixed point on its Bass-Serre tree (cf. \S 4, Theorem 7 in \cite{S77}). 
\end{itemize}
\end{example}

The following observation shows that property FA is inherited by quotients and behaves well with respect to group extensions.

\begin{observation}\label{obs:FA} \hfill \\
Let $G$ be a group and $N \subseteq G$ a normal subgroup.
\begin{itemize}
\item[$i)$] If $G$ has property FA, then so does the quotient $G/N$.
\item[$ii)$] If both $N$ and $G/N$ have property FA, then also $G$ satisfies FA.
\end{itemize}
\end{observation}
(cf. Examples 6.3.2 and 6.3.1 in \cite{S77})\\

A useful feature of property FA is that it can be checked by studying some suitable generating set of the group under consideration. This is a consequence of Helly's Theorem for Trees. 

\begin{theorem}[Helly's Theorem for Trees]\label{theorem:helly} \hfill \\
Let $T$ be a tree and $A_1, \ldots, A_k$ be subtrees of $T$ such that $A_i \cap A_j \neq \emptyset$ for all $i,j \in \{1, \ldots, k\}$. Then the intersection $\bigcap\limits_{i=1}^{k}A_i$ is non-empty.
\end{theorem}
(see \S 6, Lemma 10, p.65 in \cite{S77} or Theorem 3.2 in \cite{F13} for a more general topological statement)

\begin{corollary}\label{corollary:helly}
Let $G$ be a group acting on a tree $T$. Let $\{s_1, \ldots, s_k\}$ be a finite set of generators of $G$ such that for all $i,j \in \{1, \ldots, k\}$ the generators $s_i$ and $s_j$ have a common fixed point in $T$. Then, $G$ has a global fixed point in $T$.
\end{corollary}
\begin{proof}
Let $A_i=\mathrm{Fix}(s_i)$ be the fixed point set of $s_i$. By assumption, the $A_i$ are non-empty and form a family of subtrees with pairwise non-empty intersection. By Helly's Theorem for Trees, the intersection $\bigcap\limits_{i=1}^{k}A_i$ is non-empty.\\
Any vertex $v$ in this intersection is fixed by each generator $s_i$ and therefore fixed by any element of $G$. Thus, $G$ has a global fixed point in $T$.
\end{proof}

\subsection{Automorphisms of free products of finite cyclic groups}
In this subsection, we develope the necessary theory of automorphisms of free products. For this, we present a special generating set of $\Aut{G}$ and show that the automorphism group can be written as an iterated semi-direct product.\\
Let $m \geq 2$ and $n_1, \ldots, n_m \geq 2$ be natural numbers. Let $G$ be a free product of $m$ finite cyclic groups $\mathbb{Z}/n_1\mathbb{Z}, \ldots , \mathbb{Z}/n_m\mathbb{Z}$ and fix a presentation $$G=\langle x_1, \ldots, x_m \mid x_i^{n_i}, i=1 \ldots, m \rangle.$$ Since the $x_i$ generate $G$, any automorphism of $G$ is described by its action on $\{x_1, \ldots ,x_m\}$. The group $G$ admits the following automorphisms:
\begin{itemize}
\item[$i)$] \emph{Factor Automorphisms}: A factor automorphism $\varepsilon$ maps each $x_i$ to a power $x_i^k$ for some $k$ with $\mathrm{gcd}(k,n_i)=1$. (The term \emph{factor automorphism} is due to the fact that such a map restricts to an automorphims of each free factor $\langle x_i \rangle$.)
\item[$ii)$] \emph{Permutations}: Any permutation $\sigma \in \Sym{m}$ satisfying $n_{\sigma(i)}=n_i$ for all $i=1, \ldots, m$ induces a permutation of the $x_i$ which extends to an automorphism $\sigma \in \Aut{G}$.
\item[$iii)$] \emph{Partial Conjugations}: Let $i,j \in \{1, \ldots, m\}, i \neq j$. Then, the partial conjugation $\alpha_i^j$ is given as follows:
$$\alpha_i^j(x_k)=\begin{cases} x_jx_kx_j^{-1} &\text{ if } k=i \\
x_k &\text{ if } k \neq i \end{cases}$$
\end{itemize}

The following is Proposition 1.2 in \cite{C88} and also a special case of the Main Theorem in \cite{CG09}. (For more references on graph products and their automorphism groups see also \cite{G90} and \cite{L93}.)

\begin{lemma}\label{lemma:gen}
The group $\Aut{G}$ is generated by factor automorphisms, permutations and partial conjugations.
\end{lemma}

\begin{remark} \label{remark:structure} \hfill
\begin{itemize}
\item[a)] The set of factor automorphisms forms a subgroup $F$ of $\Aut{G}$ which is isomorphic to the direct product $$F \cong \prod\limits_{i=1}^{m} \Aut{\langle x_i \rangle} \cong \prod\limits_{i=1}^{m}\Aut{\mathbb{Z}/n_i\mathbb{Z}} \cong \prod\limits_{i=1}^{m}\mathbb{Z}/n_i\mathbb{Z}^{\times}$$
where $\mathbb{Z}/n_i\mathbb{Z}^{\times}$ denotes the group of units in the ring $\mathbb{Z}/n_i\mathbb{Z}$.
\item[b)] 
Set $\underline{n}=(n_1, \ldots, n_m)$ and let $\Sym{\underline{n}} \subseteq \Sym{m}$ be the subgroup of those permutations $\sigma$ with $n_{\sigma(i)}=n_i$ for all $i=1, \ldots, m$.\\
Then, $\Sym{\underline{n}}$ acts on $F$ via change of coordinates. These groups generate in $\Aut{G}$ a semi-direct product $F \rtimes \Sym{\underline{n}}$ since $\Sym{\underline{n}}$ normalises $F$.
\item[c)] The subgroup generated by partial conjugations is called the \emph{Fouxe-Rabino-vitch subgroup} $\mathrm{FR}(G)$.\\
Since $\prod\limits_{i \neq j} \alpha_i^j$ equals the conjugation with $x_j$ the subgroup $\mathrm{FR}(G)$ contains the group of inner automorphisms. For $m=2$, any partial conjugation is an inner automorphism and $\FR{G}$ coincides with $\Inn{G}$.
\end{itemize}
\end{remark}

The following two lemmas can be deduced from Theorem 3.1 in \cite{GPR09}. (Also see Proposition 3.7 therein.)
As there are short direct proofs, we have included them here.
  
\begin{lemma}\label{lemma:FRnormal} The Fouxe-Rabinovitch subgroup $\mathrm{FR}(G)$ is a normal subgroup of $\Aut{G}$. 
\end{lemma}
\begin{proof} Let $\alpha_i^j$ be a partial conjugation and $\varepsilon \in F, \pi \in \Sym{\underline{n}}$.\\
We write $\varepsilon=(\varepsilon_1, \ldots, \varepsilon_m)$ where $\varepsilon(x_i)=x_i^{\varepsilon_i}$ for $i=1, \ldots, m$. Then, a short calculation shows
$$ \varepsilon \alpha_i^j \varepsilon^{-1}=(\alpha_i^j)^{\varepsilon_j} \text{ and }\pi \alpha_i^j \pi^{-1}=\alpha_{\pi(i)}^{\pi(j)}.$$
Hence, $\mathrm{FR}(G)$ is normalised by factor automorphisms and permutations. By Lemma \ref{lemma:gen}, $\mathrm{FR}(G)$ is a normal subgroup of $\Aut{G}$.
\end{proof}

\begin{lemma}\label{lemma:autstructure} The automorphism group $\Aut{G}$ is isomorphic to the semi-direct product $$\mathrm{FR}(G) \rtimes (F \rtimes \Sym{\underline{n}}).$$
\end{lemma}
\begin{proof} By Remark \ref{remark:structure} b) the factor automorphisms and permutations generate a semi-direct product $F \rtimes \Sym{\underline{n}}$. By Lemma \ref{lemma:gen}, $\Aut{G}$ is generated by $\FR{G}$ and $F \rtimes \Sym{\underline{n}}$. Moreover, $\FR{G}$ is normal in $\Aut{G}$ by the previous Lemma \ref{lemma:FRnormal}. Hence, it only remains to prove that $\FR{G}$ intersects $F \rtimes \Sym{\underline{n}}$ trivially.\\
Let $\varphi \in  \FR{G} \cap (F \rtimes \Sym{\underline{n}})$ and let $i \in \{1, \ldots, m\}$ be arbitrary. Since $\varphi$ belongs to $F \rtimes \Sym{\underline{n}}$ it maps the free factor $\langle x_i \rangle$ to some free factor $\langle x_j \rangle$, i.e. $\varphi(x_i)=x_j^k$ for some suitable $k$. But $\varphi \in \FR{G}$ also implies that $\varphi(x_i)$ is conjugate to $x_i$.\\
\textbf{Claim}: $x_i$ is not conjugate to any element $y \in \langle x_j \rangle$ for $j \neq i$.\\
Let $\pi_i:G \rightarrow \langle x_i \rangle$ be the homomorphism defined by $x_i \mapsto x_i$ and $x_j \mapsto 1$ for $j \neq i$. Then, the restriction of $\pi_i$ to $\langle x_i \rangle$ is injective while all free factors $\langle x_j \rangle$ with $j \neq i$ are contained in the kernel of $\pi_i$. As $\ker(\pi_i)$ is normal in $G$ this shows the claim.\\
By the claim, $x_j^k$ only can be conjugate to $x_i$ if $i=j$ and hence $\varphi(x_i)=x_i^k$. Since $x_i$ and $x_i^k$ are conjugate in $G$, also $\pi_i(x_i)=x_i$ and $\pi_i(x_i^k)=x_i^k$ must be conjugate in $\langle x_i \rangle$. Since $\langle x_i \rangle$ is abelian, this yields $x_i^k=x_i$ and hence $\varphi(x_i)=x_i$.\\
As $i$ was arbitrary, we have $\varphi(x_j)=x_j$ for all $j$ and therefore $\varphi=\mathrm{id}$.
\end{proof}

\begin{remark}
The decomposition of $\Aut{G}$ given by the lemma corresponds to the semi-direct product structure described by Theorem 3.1 in \cite{GPR09} where $\FR{G}$ coincides with $\mathrm{Aut}^0 G$ and $F \rtimes \Sym{\underline{n}}$ with $\mathrm{Aut}^1 G$.
\end{remark}

\subsection{Characteristic Subgroups}
This subsection prepares the reduction techniques required in Section 6. Basic facts on characteristic subgroups can be found in many textbooks on group theory, e.g. \cite{S82}.
  
\begin{definition}
Let $G$ be a group. A subgroup $H \subseteq G$ is called \emph{characteristic} if $\varphi(H)=H$ for every automorphism $\varphi \in \Aut{G}$.
\end{definition}

\begin{example}
Let $G$ be a group. The following subgroups are characteristic:
\begin{itemize}
\item[a)] the trivial group $\{1\}$ and the whole group $G$
\item[b)] the center $Z(G)$
\item[c)] the commutator subgroup $DG$
\item[d)] the subgroup generated by all elements of order $k$ for some $k \in \mathbb{N} \cup \{\infty\}$
\item[e)] the subgroup generated by all subgroups isomorphic to a given group $\Gamma$ (In the case of a cyclic group $\Gamma$ this coincides with the previous example.)
\end{itemize} 
\end{example}

\begin{remark}\label{remchar}
Let $G$ be a group.
\begin{itemize} 
\item[$i)$] Since for any $g \in G$ the conjugation map $x \mapsto gxg^{-1}$ defines an automorphism of $G$, each characteristic subgroup $H$ of $G$ is normal. 
\item[$ii)$] If $H,H' \subseteq G$ are two characteristic subgroups of $G$, then also their intersection $H \cap H'$ and the subgroup $\langle H \cup H' \rangle$ are characteristic.
\item[$iii)$] Let $H \subseteq G$ be a subgroup with $\varphi(H) \subseteq H$ for all $\varphi \in \Aut{G}$. Then, $H$ is a characteristic subgroup. For a given automorphism $\varphi \in \Aut{G}$ we can consider its inverse $\varphi^{-1} \in \Aut{G}$. By assumption, we have $\varphi^{-1}(H) \subseteq H$ which yields $H \subseteq \varphi(H)$ and hence $\varphi(H)=H$.
\end{itemize}
\end{remark}

\begin{lemma}\label{lemma:inducedmap} Let $G$ be a group and $N \subseteq G$ a characteristic subgroup of $G$. Then, there is a well-defined group homomorphism $\Aut{G} \rightarrow \Aut{G/N}, \theta \mapsto \bar{\theta}$ where $\bar{\theta}(gN)=\theta(g)N$.
\end{lemma}

\begin{proof} If $\theta$ is well-defined it is immediate from the definition that it is an automorphism of $G/N$ with inverse $\overline{\theta^{-1}}$. So let $g_1, g_2 \in G$ such that $g_1N=g_2N$. Then, we have $g_1^{-1}g_2 \in N$ and thus $\theta(g_1)^{-1}\theta(g_2)=\theta(g_1^{-1}g_2) \in \theta(N)=N$.\\
This implies $\theta(g_1)N=\theta(g_2)N$ and $\bar{\theta}$ is well-defined.
\end{proof}


\subsection{Extending actions to semi-direct products}

In this subsection, we provide a lemma which we use in the proof of Theorem 1.1 in the cases $m=2$ (Section 3) and $m=3$ (Section 4). 

\begin{lemma}\label{lemma:extend}
Let $G=N \rtimes H$ be a semi-direct product and let $\mathcal{U}$ be a set of subgroups of $N$ which is $H$-invariant, i.e. for all $U \in \mathcal{U}$ and all $h \in H$ we have $hUh^{-1} \in \mathcal{U}$. Then, the left regular action of $N$ on $\bigcup \limits_{U \in \mathcal{U}}N/U$ extends to $G$.
\end{lemma}
\begin{proof}
Write $h(x)=hxh^{-1}$ for $h \in H,x \in N$. Set $h(nU):=h(n)h(U)$ for $h \in H, n \in N$ and $U \in \mathcal{U}$. This is obviously a group action if it is well-defined. By assumption, $U \in \mathcal{U}$ implies $h(U) \in \mathcal{U}$. If $nU=n'U$, we have $n^{-1}n' \in U$ and thus $h(n)^{-1}h(n')=h(n^{-1}n') \in h(U)$ and therefore $h(n)h(U)=h(n')h(U)$ which shows that the action of $H$ is well-defined.\\
The separate actions of $N$ and $H$ induce an action of the free product $N \ast H$ on $\bigcup \limits_{U \in \mathcal{U}}N/U$. In order to obtain an action of $G$, we need to check that this action is compatible with the semi-direct product structure on $G$. This means that we have to show for all $n \in N$ and $h \in H$ that the element $h(n)^{-1}hnh^{-1}\in N \ast H$ acts trivially. Let $xU \in N/U$ for some $U \in \mathcal{U}$. Then we have:
\begin{flalign*}
\begin{split}
h(n)^{-1}hnh^{-1}(xU)&=h(n)^{-1}hn(h^{-1}(x)h^{-1}(U))\\
&=h(n)^{-1}h(nh^{-1}(x)h^{-1}(U))\\
&=h(n)^{-1}(h(n)hh^{-1}(x)hh^{-1}(U))\\
&=h(n)^{-1}(h(n)xU)=xU
\end{split}
\end{flalign*}
Since $xU \in N/U$ was arbitrary, the element $h(n)^{-1}hnh^{-1}$ acts trivially on $\bigcup \limits_{U \in \mathcal{U}}N/U$ and we obtain the required action of $G$.
\end{proof}

\section{$\Aut{\mathbb{Z}/n\mathbb{Z} \ast \mathbb{Z}/n \mathbb{Z}}$ does not have FA}
The goal of this section is to prove Theorem 1.1 for $m=2$. Let $n \geq 2$ and $G=\mathbb{Z}/n\mathbb{Z} \ast \mathbb{Z}/n \mathbb{Z}$. We fix a presentation $G=\langle a,b \mid a^n, b^n \rangle$ and set ${A = \langle a \rangle,}$ $B= \langle b \rangle$. Applying Lemma \ref{lemma:autstructure} we obtain:

\begin{lemma}\label{lemma:struct2}
The group $\Aut{G}$ is isomorphic to $G \rtimes (F \rtimes \Sym{2})$ where $F \cong (\mathbb{Z}/n\mathbb{Z}^{\times})^2$ is the subgroup of factor automorphisms of $G$.
\end{lemma} 
\begin{proof}
By Remark \ref{remark:structure} c) the subgroup $\mathrm{FR}(G)$ equals $\mathrm{Inn}(G)$. Since $G$ has trivial center, we have $\mathrm{FR}(G) = \mathrm{Inn}(G) \cong G$.\\
Since both free factors have order $n$, the permutation subgroup $\Sym{\underline{n}}$ is the whole symmetric group $\Sym{2}$.
\end{proof}


\begin{proposition}
The group $\Aut{G}$ does not have property FA.
\end{proposition}
\begin{proof}In the following, we construct an action of $\mathrm{Aut}(G)$ on a tree without a global fixed point.\\
Recall that $G$ acts without a fixed point on the Bass-Serre tree $T$ which is given as follows (cf. \S 4, Theorem 7 in \cite{S77}): The vertices of $T$ are $V(T)=G/A \sqcup G/B$ and a pair of cosets $\{g_1A,g_2B\}$ forms an edge iff $g_1A \cap g_2B \neq \emptyset$, i.e. iff there exists $g \in G$ such that $g_1A=gA$ and $g_2B=gB$. (Therefore, the edges are in bijection with $G=G/\{1\}$.) The following picture shows a part of the tree $T$ in the case $n=3$:
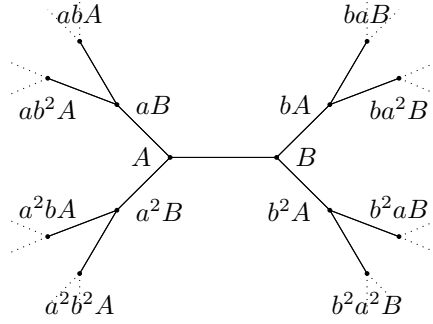
\begin{figure}[H]
\begin{center}
\begin{tikzpicture}[scale=0.7]
\node (a1) at (-1,0) [label=left:$A$]{};
\node (a2) at (1,0) [label=right:$B$]{};
\node (a3) at (-2,1) [label=right:$aB$]{};
\node (a4) at (-2,-1) [label=right:$a^2B$]{};
\node (a5) at (-3.3,1.5) [label=below:$ab^2A$]{};
\node (a6) at (-2.7,2.2) [label=$abA$]{};
\node (a7) at (-3.3,-1.5) [label=$a^2bA$]{};
\node (a8) at (-2.7,-2.2) [label=below:$a^2b^2A$]{};
\node (a9) at (2,1) [label=left:$bA$]{};
\node (a10) at (3.3,1.5) [label=below:$ba^2B$]{};
\node (a11) at (2.7,2.2) [label=$baB$]{};
\node (a12) at (3.3,-1.5) [label=$b^2aB$]{};
\node (a13) at (2.7,-2.2) [label=below:$b^2a^2B$]{};
\node (a14) at (2,-1) [label=left:$b^2A$]{};
\draw[fill=black]  (-1,0) circle (1pt);
\draw[fill=black]  (1,0) circle (1pt);
\draw[fill=black]  (-2,1) circle (1pt);
\draw[fill=black]  (-2,-1) circle (1pt);
\draw[fill=black]  (-3.3,1.5) circle (1pt);
\draw[fill=black]  (-2.7,2.2) circle (1pt);
\draw[fill=black]  (-3.3,-1.5) circle (1pt);
\draw[fill=black]  (-2.7,-2.2) circle (1pt);
\draw[fill=black]  (2,1) circle (1pt);
\draw[fill=black]  (3.3,1.5) circle (1pt);
\draw[fill=black]  (2.7,2.2) circle (1pt);
\draw[fill=black]  (3.3,-1.5) circle (1pt);
\draw[fill=black]  (2.7,-2.2) circle (1pt);
\draw[fill=black]  (2,-1) circle (1pt);
\draw[fill=black]  (-1,0) --(1, 0);
\draw[fill=black]  (-1,0)--(-2, 1);
\draw[fill=black]  (-1,0)--(-2, -1);
\draw[fill=black]  (2,1)--(3.3,1.5);
\draw[fill=black]  (2,1)--(2.7,2.2);
\draw[fill=black]  (2,-1)--(3.3,-1.5);
\draw[fill=black]  (2,-1)--(2.7,-2.2);
\draw[fill=black]  (-2,1)--(-3.3,1.5);
\draw[fill=black]  (-2,1)--(-2.7,2.2);
\draw[fill=black]  (-2,-1)--(-3.3,-1.5);
\draw[fill=black]  (-2,-1)--(-2.7,-2.2);
\draw[fill=black]  (1,0)--(2, 1);
\draw[fill=black]  (1,0)--(2, -1);
\draw[dotted]  (2.7,2.2)--(2.7, 3);
\draw[dotted]  (2.7,2.2)--(3.3, 2.8);
\draw[dotted]  (-2.7,2.2)--(-2.7, 3);
\draw[dotted]  (-2.7,2.2)--(-3.3, 2.8);
\draw[dotted]  (2.7,-2.2)--(2.7, -3);
\draw[dotted]  (2.7,-2.2)--(3.3, -2.8);
\draw[dotted]  (-2.7,-2.2)--(-2.7, -3);
\draw[dotted]  (-2.7,-2.2)--(-3.3, -2.8);
\draw[dotted]  (3.3,1.5)--(4, 1.8);
\draw[dotted]  (3.3,1.5)--(4, 1.2);
\draw[dotted]  (3.3,-1.5)--(4, -1.8);
\draw[dotted]  (3.3,-1.5)--(4, -1.2);
\draw[dotted]  (-3.3,1.5)--(-4, 1.8);
\draw[dotted]  (-3.3,1.5)--(-4, 1.2);
\draw[dotted]  (-3.3,-1.5)--(-4, -1.8);
\draw[dotted]  (-3.3,-1.5)--(-4, -1.2);
\end{tikzpicture}
\end{center}
\end{figure}
The action of $G$ on $T$ is given by 
left multiplication. Since $\mathrm{Inn}(G)$ is isomorphic to $G$, we can view the action of $G$ as an action of $\mathrm{Inn}(G) \subseteq \mathrm{Aut}(G)$.\\
Our goal is now to extend this action to the whole automorphism group $\mathrm{Aut}(G)$.
The factor automorphisms in $\mathrm{Aut}(A) \times \mathrm{Aut}(B)$ map $A$ to $A$ and $B$ to $B$. The generator $\sigma \in \mathrm{Sym}(2)$ interchanges $A$ and $B$. Thus, $F \rtimes \Sym{2}$ preserves the set of subgroups $\{A,B\}$ and we can apply Lemma \ref{lemma:extend} and obtain an action of $\Aut{G}=G \rtimes (F \rtimes \Sym{2})$ on $V(T)=G/A \sqcup G/B$.\\ 
It remains to check that edges are mapped to edges. Let $e \in E(T)$ be an edge of $T$. Then, $e$ has end-points $g_1A,g_2B$ with $g_1A \cap g_2B\neq \emptyset$. As mentioned above, we find $g \in G$ such that $gA=g_1A$ and $gB=g_2B$. This yields:
\begin{flalign*}
\begin{split}
\varphi(g_1A)&=\varphi(gA)=\varphi(g)\varphi(A)\\
\varphi(g_2B)&=\varphi(gB)=\varphi(g)\varphi(B)
\end{split}
\end{flalign*}
That shows $\varphi(g) \in \varphi(g_1A) \cap \varphi(g_2B)$, i.e $\varphi(g_1A) \cap \varphi(g_2B) \neq \emptyset$ and $\varphi(g_1A), \varphi(g_2B)$ are joined by an edge. Thus, $\Aut{G}$ acts by graph automorphisms on the tree $T$. Since $G$ acts without global fixed point on $T$, also the extended action has no global fixed point.\\
By passing to the barycentric subdivision $\mathrm{sd}(T)$, we get rid of the edge inversions (which were induced by $\sigma\in \mathrm{Sym}(2)$). The action of $\mathrm{Aut}(G)$ on $\mathrm{sd}(T)$ shows that $\mathrm{Aut}(G)$ does not have property FA.
\end{proof}

\begin{remark}\label{remark:amalgam2}
The proof yields also a possibility to present $\Aut{G}$ as an amalgamated product.\\
Since $G$ acts transitively on the edges of $T$ and the generator $\sigma \in \mathrm{Sym}(2)$ maps the edge $(A \subseteq \{A,B\}) \in E(\mathrm{sd}(T))$ to the edge $(B \subseteq \{A,B\})$, $\mathrm{Aut}(G)$ acts transitively on the edges of $\mathrm{sd}(T)$. Moreover, any automorphism which stabilises an edge $e \in E(\mathrm{sd}(T))$ has to fix both its endpoints, i.e. has to fix $e$ pointwise. Hence, the edge $(A \subseteq \{A,B\})$ forms a fundamental domain for the action of $\mathrm{Aut}(G)$ on $\mathrm{sd}(T)$.\\
By Bass-Serre theory (cf. \S 4, Theorem 6 in \cite{S77}), $\mathrm{Aut}(G)$ is isomorphic to the free product of the vertex stabilisers $\mathrm{Aut}(G)_A$ and $\mathrm{Aut}(G)_{\{A,B\}}$ amalgamated along the stabiliser of the whole edge $\mathrm{Aut}(G)_{(A \subseteq \{A,B\})}$. With Lemma \ref{lemma:struct2} one can easily determine these stabilisers. One has $\Aut{G}_A=A \rtimes F$, ${\Aut{G}_{\{A,B\}}=F \rtimes \Sym{2}}$ and hence
$$\Aut{G} \cong (A \rtimes F) \ast_F (F \rtimes \mathrm{Sym}(2)).$$
\end{remark}

\begin{remark}
If $n_1,n_2 \geq 2$ with $n_1 \neq n_2$ and $G=\mathbb{Z}/n_1 \mathbb{Z} \times \mathbb{Z}/n_2 \mathbb{Z}$, we can repeat the above proof of Proposition 3.2. The only difference is that $\Sym{n_1,n_2}=\{\mathrm{id}\}$ here is trivial. (So we do not need to pass to the barycentric subdivision.) 
\end{remark}

\section{$\Aut{\mathbb{Z}/n\mathbb{Z} \ast \mathbb{Z}/n\mathbb{Z} \ast \mathbb{Z}/n \mathbb{Z}}$ does not have FA}

This section is devoted to prove Theorem 1.1 for $m=3$. The main strategy is to understand the outer automorphism group $\Out{G}=\Aut{G}/\mathrm{Inn}(G)$ and construct an action of $\Out{G}$ on a tree without a global fixed point.

Assume that $G$ is a free product of three copies of $\mathbb{Z}/n\mathbb{Z}$ for some $n \geq 2$. We fix a presentation $G=\langle x_1, x_2, x_3 \mid x_i^n, i \in  \{1,2,3\} \rangle$. First, we investigate the structure of the Fouxe-Rabinovitch subgroup $\mathrm{FR}(G)$.

\begin{lemma}\label{lemma:struct3}
The Fouxe-Rabinovitch subgroup $\mathrm{FR}(G)$ is isomorphic to $G \rtimes G$.
\end{lemma}
This semi-direct product structure is a special case of Theorem 3.2 in \cite{GPR09}. The proof below is based on the presentation of $\FR{G}$ given in $\cite{CG90}$.
\begin{proof}
By definition, $\mathrm{FR}(G)$ is generated by the six partial conjugations $\alpha_1^2,\alpha_1^3$, $\alpha_2^1,\alpha_2^3, \alpha_3^1$ and $\alpha_3^2$.\\
Let $c_y:G \rightarrow G$ be the conjugation homomorphism $x \mapsto yxy^{-1}$ and denote by $c_i:=c_{x_i}$ the conjugation with the generator $x_i$ for $i \in \{1,2,3\}$.\\
Since the conjugation $c_1$ is the product $\alpha_2^1 \alpha_3^1$ (and for $i=2,3$ resp.) also the set $X=\{c_1, c_2, c_3, \alpha_1^2, \alpha_2^3, \alpha_3^1\}$ forms a generating set for $\mathrm{FR}(G)$. Since $G$ has trivial center, we have $\langle c_1, c_2, c_3 \rangle = \mathrm{Inn}(G) \cong G$. This is a normal subgroup of $\mathrm{FR}(G)$ (since it is normal in the whole automorphism group $\Aut{G}$) and the conjugation action is given as follows:\\
For $i=1,2,3$ the automorphism $\varphi c_i \varphi^{-1}$ equals $c_{\varphi(x_i)}$, i.e. the inner automorphism conjugating with $\varphi(x_i)$. It remains to be shown that $\mathrm{FR}(G)$ is isomorphic to the semidirect product of $\mathrm{Inn}(G)$ and the subgroup $\langle \alpha_1^2, \alpha_2^3, \alpha_3^1 \rangle \cong G$ with respect to the described action.\\
Let $\FR{G}'$ be the group with the presentation consisting of the generating set $X$ with the relations $x^n=1$ for $x \in X$ and $\alpha_i^j c_k (\alpha_i^j)^{-1}=c_{\alpha_i^j(x_k)}$ for $i,j,k \in \{1,2,3\}$ where the right hand side is written as a word in the generators $c_i$, e.g. $\alpha_1^2 c_1 (\alpha_1^2)^{-1}=c_2c_1c_2^{-1}$.\\
By construction, we have $\FR{G}' \cong G \rtimes G$ with respect to the action above.\\
By Proposition 3.1 in \cite{CG90}, $\mathrm{FR}(G)$ has a presentation where the generators are the partial conjugations with the relations:
\begin{itemize}
\item[$i)$] $(\alpha_i^j)^m=1$ for all $i \neq j$
\item[$ii)$] $[\alpha_i^j, \alpha_k^j]=1$ for $j=1,2,3$ and $i,k \neq j$
\item[$iii)$] $[\alpha_i^j \alpha_k^j, \alpha_i^k]=1$ for pairwise distinct $i,j,k$
\end{itemize}
In the following, we construct homomorphisms to show that these two presentations are isomorphic.
To distinguish clearly between them we introduce a new notation (which is also used by Collins, Gilbert in \cite{CG90}):\\ Let $X_i= \langle x_i \rangle$ be the free factor generated by $x_i$ and denote by $(X_i,x_j)$ the partial conjugation where the free factor $X_i$ is conjugated by the letter $x_j$. Thus, $(X_i,x_j)$ corresponds to $\alpha_i^j$ in our usual notation.\\
The relations in the presentation of \cite{CG90} hence are given as:
\begin{itemize}
\item[$i)$] $(X_i,x_j)^m=1$ for all $i \neq j$
\item[$ii)$] $[(X_i,x_j),(X_k,x_j)]=1$ for $j=1,2,3$ and $i,k \neq j$
\item[$iii)$] $[(X_i,x_j) (X_k,x_j), (X_i,x_k)]=1$ for pairwise distinct $i,j,k$
\end{itemize}
Define a homomorphism $\varphi:\mathrm{FR}(G)' \rightarrow \mathrm{FR}(G)$ by the following asignment: 
\begin{flalign*}
\begin{split}
&c_1 \mapsto (X_2,x_1) (X_3,x_1), ~~~ \alpha_1^2 \mapsto (X_1,x_2)\\ 
&c_2 \mapsto (X_1,x_2)(X_3,x_2), ~~~ \alpha_2^3 \mapsto (X_2,x_3)\\
&c_3 \mapsto (X_1,x_3)(X_2,x_3), ~~~ \alpha_3^1 \mapsto (X_3,x_1)
\end{split}
\end{flalign*}
To see that $\varphi$ indeed gives a well-defined homomorphism, one can check that the relations in $\FR{G}'$ are preserved by $\varphi$.\\
Vice versa we define a homomorphism $\psi:\mathrm{FR}(G) \rightarrow \mathrm{FR}(G)'$ by the following asignment: 
\begin{flalign*}
\begin{split}
&(X_1,x_2) \mapsto \alpha_1^2 , ~~~ (X_1,x_3) \mapsto c_3 (\alpha_2^3)^{-1}\\ 
&(X_2,x_1) \mapsto c_1 (\alpha_1^3)^{-1} , ~~~ (X_2,x_3) \mapsto \alpha_2^3\\
&(X_3,x_1) \mapsto \alpha_3^1, ~~~ (X_3,x_2) \mapsto c_2 (\alpha_1^2)^{-1}
\end{split}
\end{flalign*}
Again, one checks by some short calculations that $\psi$ preserves the relations in the presentation of $\FR{G}$. As an example, for the relation $iii)$ we have to show that the images of $(X_1,x_2)(X_3,x_2)$ and $(X_1,x_3)$ commute. We have:
\begin{flalign*}
\begin{split}
\psi(X_1,x_2) \psi(X_3,x_2) \psi(X_1,x_3)&=\alpha_1^2 c_2 (\alpha_1^2)^{-1} c_3 (\alpha_2^3)^{-1}\\
&=c_2 c_3 (\alpha_2^3)^{-1}=c_2 (\alpha_2^3)^{-1} c_3 \\
&=(\alpha_2^3)^{-1} (\alpha_2^3 c_2 (\alpha_2^3)^{-1}) c_3 \\
&= (\alpha_2^3)^{-1} (c_3 c_2 c_3^{-1}) c_3 = (\alpha_2^3)^{-1} c_3 c_2 \\
&=c_3 (\alpha_2^3)^{-1} c_2 = c_3 (\alpha_2^3)^{-1} \alpha_1^2 c_2 (\alpha_1^2)^{-1}\\
&=\psi(X_1,x_3) \psi(X_1,x_2) \psi(X_3,x_2)
\end{split}
\end{flalign*} 
By construction, the homomorphisms $\varphi$ and $\psi$ are mutually inverse. Thus, the two presentations are isomorphic and we have $\FR{G} \cong \FR{G}' \cong G \rtimes G$.
\end{proof}

With Lemma \ref{lemma:autstructure}, we obtain the following corollary.

\begin{corollary}\label{cor:3}
The group $\Aut{G}$ is isomorphic to $(G \rtimes G) \rtimes (F \rtimes \Sym{3})$ where ${F \cong (\mathbb{Z}/n\mathbb{Z}^{\times})^3}$ is the subgroup of factor automorphisms of $G$.\\
In particular, the group $\Out{G}$ is isomorphic to $G \rtimes (F \rtimes \Sym{3})$.
\end{corollary} 

\begin{remark}\label{remark:tree3}
By Lemma \ref{lemma:struct3} the subgroup $\langle \alpha_1^2, \alpha_2^3, \alpha_3^1 \rangle \subseteq \FR{G}$ maps isomorphically onto $\FR{G}/\mathrm{Inn}(G) \cong G$. So we identify $G$ with $\langle \alpha_1^2, \alpha_2^3, \alpha_3^1 \rangle$. It acts on the associated Bass-Serre tree $T$ (cf. \S 4.5 in \cite{S77}): Let $A_1= \langle \alpha_1^2 \rangle, A_2= \langle \alpha_2^3 \rangle$ and $A_3= \langle \alpha_3^1 \rangle$. Note that $G$ has the following decomposition as a graph of groups
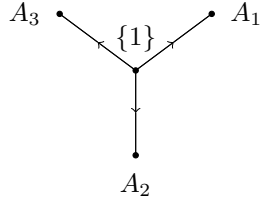
\begin{figure}[H]
\begin{center}
\begin{tikzpicture}[scale=0.5]
\node (a1) at (0,0) [label=$\{1\}$]{};
\node (a2) at (2,1.5) [label=right:$A_1$]{};
\node (a3) at (-2,1.5) [label=left:$A_3$]{};
\node (a4) at (0,-2.25) [label=below:$A_2$]{};
\draw[fill=black]  (0,0) circle (2pt);
\draw[fill=black]  (2,1.5) circle (2pt);
\draw[fill=black]  (-2,1.5) circle (2pt);
\draw[fill=black]  (0,-2.25) circle (2pt);
\draw[->, fill=black]  (0,0) -- (1,0.75);
\draw[fill=black]  (1,0.75)--(2,1.5);
\draw[->, fill=black]  (0,0)--(-1,0.75);
\draw[fill=black]  (-1,0.75)--(-2,1.5);
\draw[->, fill=black]  (0,0)--(0,-1.125);
\draw[fill=black]  (0,-1.125)--(0,-2.25);
\end{tikzpicture}
\end{center}
\end{figure}
\noindent where all edge homomorphisms are trivial. Thus, the Bass-Serre tree $T$ is given as follows: The set of vertices of $T$ is the disjoint union of $G=G/\{1\}$, $G/A_1$, $G/A_2$ and $G/A_3$. A vertex $g \in G$ forms an edge with a  coset $hA_i$ if $g \in hA_i$.\\
The group $G$ acts on $T$ by left multiplication. The vertex stabilisers are either trivial (for vertices $g \in G$) or conjugate to an $A_i$. In particular, $G$ acts on $T$ without a global fixed point.
\end{remark}

\begin{remark}
By Corollary \ref{cor:3}, we have $\mathrm{Out}(G)\cong G \rtimes (F \rtimes \Sym{3})$. The action of $F \rtimes \Sym{3}$ on $G$ is given as follows:\\
Let $\varepsilon \in F$ be a factor automorphism mapping $x_k$ to the power $x_k^{\varepsilon_k}$. Then, we have $\varepsilon \alpha_i^{i+1} \varepsilon^{-1}=(\alpha_i^{i+1})^{\varepsilon_{i+1}}$. For a permutation $\pi \in \Sym{3}$ we find $\pi \alpha_i^{i+1} \pi^{-1}=\alpha_{\pi(i)}^{\pi(i+1)}$ which has to be interpreted in the quotient. For example, $(12) \alpha_1^2 (12)=\alpha_2^1=(\alpha_3^1)^{-1}$ in $\Aut{G}/\mathrm{Inn}(G)$.\\
The permutation group $\Sym{3}$ acts on the set $\{(\alpha_1^2)^{\pm 1},(\alpha_2^3)^{\pm 1},(\alpha_3^1)^{\pm 1}\}$ by signed permutations.
\end{remark}

\begin{proposition}
The group $\Aut{G}$ does not have property FA.
\end{proposition}
\begin{proof}
Analogously to the proof of Proposition 3.2, we construct an action of $\Out{G}$ on the tree $T$ described in Remark \ref{remark:tree3} by extending the action of ${G= \langle \alpha_1^2, \alpha_2^3, \alpha_3^1 \rangle}$.\\
By the previous Remark 4.4, the subgroup $F \rtimes \Sym{3}$ preserves the set of free factors $\{A_1,A_2,A_3\}$ and fixes the trivial subgroup $\{1\}$. Thus, we can apply Lemma \ref{lemma:extend} to extend the action of $G$ on $V(T)$ to an action of $\Out{G}$. It is easy to check from the definition that $\Out{G}$ maps edges to edges and we obtain a well-defined action of $\Out{G}$ on the tree $T$.\\
Note that $T$ is bipartite. Each vertex is of element or coset type and $\Out{G}$ acts type-preserving. Therefore, the $\Out{G}$-action is without inversion.\\
By the description of the stabilisers of the $G$-action in Remark \ref{remark:tree3}, there is no global fixed point and hence $\Out{G}$ does not have FA. As $\Out{G}$ is a quotient of $\Aut{G}$ and FA is preserved by quotients by Observation \ref{obs:FA} $i)$, this implies that $\Aut{G}$ does not have property FA.
\end{proof}

\begin{remark}\label{remark:amalgam3}
Analogously to Remark \ref{remark:amalgam2} the proof allows us to write $\Out{G}$ as an amalgamated product.\\
Since $G$ acts transitively on the vertices of element type and $\mathrm{Sym}(3)$ acts transitively on the set of free factors $\{A_1,A_2,A_3\}$ the edge $\{1, A_1\}$ forms a fundamental domain for the action of $\Out{G}$ on $T$.\\
By Bass-Serre theory (cf. \S 4, Theorem 6 in \cite{S77}), $\Out{G}$ is isomorphic to the free product of the vertex stabilisers $\Out{G}_1$ and $\Out{G}_{A_1}$ amalgamated along the stabiliser of the whole edge $\Out{G}_{\{1,A_1\}}$.\\
With Corollary \ref{cor:3} one can again calculate the stabilisers and obtains \\
${\Out{G}_1=F \rtimes \Sym{3},}$ ${\Out{G}_{A_1}=A_1 \rtimes (F \rtimes \Sym{2})}$ and therefore
$$\Out{G} \cong F \rtimes \Sym{3} \ast_{F \rtimes \Sym{2}} A_1 \rtimes (F \rtimes \Sym{2})$$
where $\Sym{2}$ is identified with the subgroup of $\Sym{3}$ generated by the transposition $(1 \ 2)$.
\end{remark}

\section{$\Aut{\ast_{i=1}^m \mathbb{Z}/n\mathbb{Z}}$ has property FA for $m \geq 4$}

In this section, we prove the affirmative part of Theorem 1.1. Assume that $G$ is a free product of $m \geq 4$ copies of $\mathbb{Z}/n\mathbb{Z}$ for some $n \geq 2$. Fix a presentation $${G=\langle x_1, \ldots, x_m \mid x_i^n, i \in  \{1,\ldots, m\} \rangle}.$$
The main step in the proof is to show that given any action of $\Aut{G}$ on a tree $T$ the Fouxe-Rabinovitch subgroup $\FR{G}$ must fix a vertex of $T$. For this we need the \emph{Subtree Cycle Lemma} which gives conditions under which there exists a certain pair of subtrees with non-empty intersection.

In our application, the considered subtrees will be fixed point sets of partial conjugations. With Helly's Theorem for trees (Theorem 2.4) this yields the required global fixed point of $\FR{G}$.

First, we introduce some notation.

\begin{definition} Let $k \in \mathbb{N}$. For $i,j \in \{1, \ldots, k\}$ we define the \emph{cyclic distance} of $i,j$ as
$$d_c(i,j):= \min\{\lvert x- y \rvert \mid x \in i+k\mathbb{Z}, y \in j+k\mathbb{Z}\}.$$
\end{definition}

The following lemma is crucial for the proof of the case $m \geq 4$ in Theorem 1.1. We apply it in Corollary \ref{corollary:diagonal} to obtain a technique for determining fixed point sets with non-empty intersection. 

\begin{lemma}[Subtree Cycle Lemma] Let $T$ be a tree with subtrees $A_1, \ldots, A_k,$ ${k \geq 4.}$ If $A_i \cap A_j \neq \emptyset$ for all $i,j\in \{1, \ldots, k\}$ with $d_c(i,j)=1$ then there exist $i,j \in \{1, \ldots, k\}$ such that $d_c(i,j) \geq 2$ and $A_i \cap A_j \neq \emptyset$.
\end{lemma}
\begin{proof} Assume for contradiction that $A_i \cap A_j = \emptyset$ for all $i,j \in \{1, \ldots, k\}$ with $d_c(i,j) \geq 2$. Pick points $x_{i,i+1} \in A_i \cap A_{i+1}$ and $x_{k,1} \in A_k \cap A_1$ such that $\sum\limits_{i=1}^{k} d(x_{i,i+1},x_{i+1,i+2})$ is minimal (subscripts are taken$\mod k$).\\
Let $p_i$ be the shortest edge path from $x_{i-1,i}$ to $x_{i,i+1}$ and $c$ be the cycle starting at $x_{k,1}$ consisting of the paths $p_1, \ldots ,p_k$.\\
\textbf{Claim}: $c$ is a reduced edge path in $T$.\\
If $c$ is not reduced, it must contain an edge $e$ and its opposite $\bar{e}$ as consecutive edges. Since all the $p_i$ were chosen to have minimal length those consecutive edges can not belong to one of the $p_i$. Therefore, there must exist $i \in \{1, \ldots, k\}$ such that $p_{i-1}$ ends with $e$ and $p_i$ begins with the opposite edge $\bar{e}$. Let $v$ be the initial vertex of $e$ (which also is the terminal vertex of the edge $\bar{e}$).\\
Since $T$ is uniquely geodesic and the $A_j$ are convex subspaces, each path $p_j$ is contained in the subtree $A_j$. The vertex $v$ lies on the path $p_{i-1}$ and hence is contained in $A_{i-1}$. As the end point of $\bar{e}$ it also lies on the path $p_i$ which implies $v \in A_i$. This yields $v \in A_{i-1} \cap A_i$. By construction, we have $$d(x_{i-2,i-1},v)=d(x_{i-2,i-1},x_{i-1,i})-1 \text{ and } d(v,x_{i,i+1}) =d(x_{i-1,i},x_{i,i+1})-1.$$
But this contradicts the minimality of $\sum\limits_{i=1}^{k} d(x_{i,i+1},x_{i+1,i+2})$. Therefore no path $p_i$ can start with the opposite edge of the terminal edge of $p_{i-1}$. Thus, the cycle $c$ is reduced which is a contradiction to $T$ being a tree. $~~\lightning
$\\
So the assumption $A_i \cap A_j = \emptyset$ for all $i,j \in \{1, \ldots, k\}$ with $d_c(i,j) \geq 2$ was wrong and we obtain the statement of the lemma.
\end{proof}

\noindent
\textbf{Special Case}(Diagonal Lemma): Let $T$ be a tree with subtrees $A_1,A_2,B_1,B_2$. If $A_i \cap B_j \neq \emptyset$ for all $i,j\in \{1, 2\}$ then either $A_1 \cap A_2 \neq \emptyset$ or $B_1 \cap B_2 \neq \emptyset$.

\begin{remark}
The Diagonal Lemma is also a special case of the colorful Helly's Theorem (see Theorem 3.5 in \cite{D17}) which is a generalisation of Helly's Theorem (Theorem 2.4).
\end{remark}

The next Corollary specifies the Diagonal Lemma to the situation of fixed point sets of elements of a group which acts on a tree.

\begin{corollary}\label{corollary:diagonal} Assume a group $G$ acts on a tree $T$ and that there are elements $a_1,a_2,b_1,b_2 \in G$ such that $a_i$ and $b_j$ have a common fixed point for all ${i,j \in \{1,2\}.}$ Then, either $a_1$ and $a_2$ or $b_1$ and $b_2$ have a common fixed point.
\end{corollary}
\begin{proof} Apply the Diagonal Lemma to the subtrees $A_i= \mathrm{Fix}(a_i)$ and ${B_i=\mathrm{Fix}(b_i)}$ for $i=1,2$.
\end{proof}

\begin{remark}
The Diagonal Lemma can be visualised as follows: Construct a graph by taking four vertices labelled by $A_1,A_2,B_1$ and $B_2$ and inserting an edge between two vertices if their label subtrees have non-empty intersection. We obtain the following square graph:
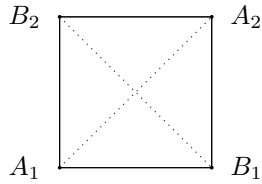
\begin{figure}[H]
\begin{center}
\begin{tikzpicture}[scale=0.5]
\node (a1) at (2,2) [label=right:$A_2$]{};
\node (a2) at (2,-2) [label=right:$B_1$]{};
\node (a3) at (-2,-2) [label=left:$A_1$]{};
\node (a4) at (-2,2) [label=left:$B_2$]{};
\draw[fill=black]  (2,2) circle (1pt);
\draw[fill=black]  (2,-2) circle (1pt);
\draw[fill=black]  (-2,2) circle (1pt);
\draw[fill=black]  (-2,-2) circle (1pt);
\draw[fill=black]  (2,2)--(-2, 2);
\draw[fill=black]  (2,2)--(2, -2);
\draw[fill=black]  (-2,-2)--(-2, 2);
\draw[fill=black]  (-2,-2)--(2, -2);
\draw[dotted]  (-2,-2)--(2, 2);
\draw[dotted]  (-2,2)--(2, -2);
\end{tikzpicture}
\end{center}
\end{figure}
Then, the Diagonal Lemma states that either $A_1 \cap A_2$ or $B_1 \cap B_2$ must be non-empty. Thus, in our diagram must \emph{exist} one of the dotted diagonals.\\
The situation of the Corollary shall be visualised by the corresponding diagram where we label the vertices by the elements $a_1,a_2,b_1,b_2 \in G$ and draw a line if their fixed point sets have non-empty intersection.
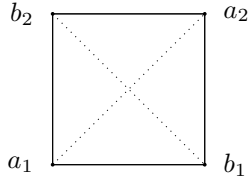
\begin{figure}[H]
\begin{center}
\begin{tikzpicture}[scale=0.5]
\draw[fill=black]  (2,2) circle (1pt);
\draw[fill=black]  (2,-2) circle (1pt);
\draw[fill=black]  (-2,2) circle (1pt);
\draw[fill=black]  (-2,-2) circle (1pt);
\draw[fill=black]  (2,2)--(-2, 2);
\draw[fill=black]  (2,2)--(2, -2);
\draw[fill=black]  (-2,-2)--(-2, 2);
\draw[fill=black]  (-2,-2)--(2, -2);
\draw[dotted]  (-2,-2)--(2, 2);
\draw[dotted]  (-2,2)--(2, -2);
\node (a1) at (2,2) [label=right:$a_2$]{};
\node (a2) at (2,-2) [label=right:$b_1$]{};
\node (a3) at (-2,-2) [label=left:$a_1$]{};
\node (a4) at (-2,2) [label=left:$b_2$]{};
\end{tikzpicture}
\end{center}
\end{figure}
\end{remark}

\begin{theorem}\label{theorem:mgeq4}
Let $m \geq 4, n  \geq 2$ and $G=\ast_{i=1}^m \mathbb{Z}/n\mathbb{Z}$ be the free product of $m$ copies of $\mathbb{Z}/n\mathbb{Z}$. Then, $\Aut{G}$ has FA.
\end{theorem}
\begin{proof}Let $\mathrm{Aut}(G)$ act on a tree $T$. We first show that the subgroup $\mathrm{FR}(G)$ generated by partial conjugations has a global fixed point in $T$.
For this we show the following: Any two partial conjugations $\alpha_i^{j}$ and $\alpha_{r}^s$ have a common fixed point.\\
Note that each partial conjugation $\alpha_{i}^j$ has finite order $n$ and therefore must fix a vertex $v \in T$. If we have $s=j$ or $\{i,j\} \cap \{r,s\} = \emptyset$, then the partial conjugations $\alpha_i^j$ and $\alpha_r^s$ commute and therefore have a common fixed point in $T$. Otherwise we can distinguish between three different cases:\\
\begin{itemize}
\item[Case 1:] $i=r$ and $s \neq j$\\
Since $m \geq 4$ there exists $l \in \{1, \ldots m\} \setminus \{i,j,s\}$. Now we apply the previous corollary to the elements $a_1=\alpha_i^j, a_2=\alpha_i^s$ and $b_1=\alpha_l^j, b_2=\alpha_l^s$. The elements $a_k$ and $b_k$ commute for $k=1,2$ because they have the same operating letter ($x_j$ for $k=1$ resp. $x_s$ for $k=2$). Since $\{i,j\} \cap \{l,s\}=\emptyset$ and $\{i,s\} \cap \{j,l\}=\emptyset$ also $a_1$ commutes with $b_2$ and $b_1$ commutes with $a_2$. Thus, $a_k$ and $b_{k'}$ have a common fixed point for $k,k' \in \{1,2\}$.
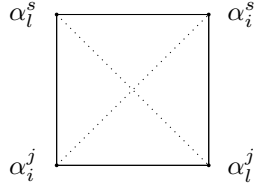
\begin{figure}[H]
\begin{center}
\begin{tikzpicture}[scale=0.5]
\draw[fill=black]  (2,2) circle (1pt);
\draw[fill=black]  (2,-2) circle (1pt);
\draw[fill=black]  (-2,2) circle (1pt);
\draw[fill=black]  (-2,-2) circle (1pt);
\draw[fill=black]  (2,2)--(-2, 2);
\draw[fill=black]  (2,2)--(2, -2);
\draw[fill=black]  (-2,-2)--(-2, 2);
\draw[fill=black]  (-2,-2)--(2, -2);
\draw[dotted]  (-2,-2)--(2, 2);
\draw[dotted]  (-2,2)--(2, -2);
\node (a1) at (2,2) [label=right:$\alpha_i^s$]{};
\node (a2) at (2,-2) [label=right:$\alpha_l^j$]{};
\node (a3) at (-2,-2) [label=left:$\alpha_i^j$]{};
\node (a4) at (-2,2) [label=left:$\alpha_l^s$]{};
\end{tikzpicture}
\end{center}
\end{figure} 
By Corollary 5.4, either $\mathrm{Fix}(a_1) \cap \mathrm{Fix}(a_2)$ or ${\mathrm{Fix}(b_1) \cap \mathrm{Fix}(b_2) }$ is non-empty. Let $\pi \in \mathrm{Aut}(G)$ be the automorphism induced by the transposition of the generators $x_i$ and $x_l$. Then we have $b_k=\pi a_k \pi^{-1}$ and hence ${\mathrm{Fix}(b_k)=\pi\mathrm{Fix}(a_k)}$  for $k=1,2$. Therefore $\mathrm{Fix}(a_1) \cap \mathrm{Fix}(a_2)$  is non-empty if and only if $\mathrm{Fix}(b_1) \cap \mathrm{Fix}(b_2)$  is non-empty. Since one of those intersections must be non-empty, both of them are non-empty and $\alpha_i^j$ and $\alpha_i^s$ have a common fixed point in $T$.
\item[Case 2:] $i=s$ and $j=r$\\
Since $m \geq 4$ we can choose $k,l \in \{1, \ldots, m\} \setminus \{i,j\}, k \neq l$. Here we apply the previous corollary to the elements $a_1=\alpha_i^j, a_2=\alpha_j^i$ and ${b_1=\alpha_k^l,}$ $b_2=\alpha_l^k.$ Since $\{i,j\} \cap \{k,l\} = \emptyset$ the elements $a_1,a_2$ commute with $b_1$ and $b_2$. Thus, $a_p$ and $b_{p'}$ have a common fixed point for $p,p' \in \{1,2\}$.
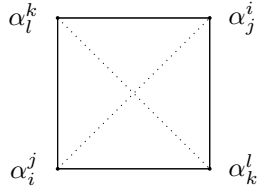
\begin{figure}[H]
\begin{center}
\begin{tikzpicture}[scale=0.5]
\draw[fill=black]  (2,2) circle (1pt);
\draw[fill=black]  (2,-2) circle (1pt);
\draw[fill=black]  (-2,2) circle (1pt);
\draw[fill=black]  (-2,-2) circle (1pt);
\draw[fill=black]  (2,2)--(-2, 2);
\draw[fill=black]  (2,2)--(2, -2);
\draw[fill=black]  (-2,-2)--(-2, 2);
\draw[fill=black]  (-2,-2)--(2, -2);
\draw[dotted]  (-2,-2)--(2, 2);
\draw[dotted]  (-2,2)--(2, -2);
\node (a1) at (2,2) [label=right:$\alpha_j^i$]{};
\node (a2) at (2,-2) [label=right:$\alpha_k^l$]{};
\node (a3) at (-2,-2) [label=left:$\alpha_i^j$]{};
\node (a4) at (-2,2) [label=left:$\alpha_l^k$]{};
\end{tikzpicture}
\end{center}
\end{figure} 
By Corollary 5.4, either $\mathrm{Fix}(a_1) \cap \mathrm{Fix}(a_2)$ or ${\mathrm{Fix}(b_1) \cap \mathrm{Fix}(b_2) }$ is non-empty. We conclude with the same argument as in case 1 by taking ${\pi \in \mathrm{Aut}(G)}$ to be the automorphism induced by interchanging $x_i$ with $x_k$ and $x_j$ with $x_l$.
\item[Case 3:] $i \neq s$ and $j=r$\\
We use the previous corollary in the situation $a_1=\alpha_i^j, a_2=\alpha_j^s$ and $b_1=\alpha_j^i, b_2=\alpha_s^j$. The elements $a_k$ and $b_k$ have a common fixed point for $k=1,2$ by case 2. Moreover, $b_1=\alpha_j^i$ has a common fixed point with $a_2=\alpha_j^s$ by Case 1. Finally, $a_1=\alpha_i^j$ and $b_2=\alpha_s^j$ have the same operating letter $x_j$. Therefore, $a_1$ and $b_2$ commute and hence have a common fixed point.
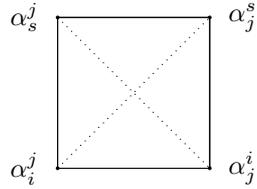
\begin{figure}[H]
\begin{center}
\begin{tikzpicture}[scale=0.5]
\draw[fill=black]  (2,2) circle (1pt);
\draw[fill=black]  (2,-2) circle (1pt);
\draw[fill=black]  (-2,2) circle (1pt);
\draw[fill=black]  (-2,-2) circle (1pt);
\draw[fill=black]  (2,2)--(-2, 2);
\draw[fill=black]  (2,2)--(2, -2);
\draw[fill=black]  (-2,-2)--(-2, 2);
\draw[fill=black]  (-2,-2)--(2, -2);
\draw[dotted]  (-2,-2)--(2, 2);
\draw[dotted]  (-2,2)--(2, -2);
\node (a1) at (2,2) [label=right:$\alpha_j^s$]{};
\node (a2) at (2,-2) [label=right:$\alpha_j^i$]{};
\node (a3) at (-2,-2) [label=left:$\alpha_i^j$]{};
\node (a4) at (-2,2) [label=left:$\alpha_s^j$]{};
\end{tikzpicture}
\end{center}
\end{figure}
We argue again as in case 1 taking $\pi \in \mathrm{Aut}(G)$ to be the automorphism induced by the transposition of the generators $x_i$ and $x_s$.
\end{itemize}
So any two partial conjugations $\alpha_i^{j}$ and $\alpha_{r}^s$ have a common fixed point. By Corollary \ref{corollary:helly}, this implies that the subgroup $\mathrm{FR}(G)$ fixes a vertex in $T$.\\
Since $\mathrm{FR}(G)$ is normalised by factor automorphisms and permutations by Lemma \ref{lemma:FRnormal}, the subgroup $F \rtimes \Sym{m} \subseteq \Aut{G}$  stabilises the fixed point tree $T^{\mathrm{FR}(G)}$. Since $F \rtimes \Sym{m}$ is finite, it fixes a vertex $v$ in $T^{\mathrm{FR}(G)}$. Then, $v$ is a global fixed point of $\mathrm{Aut}(G)$ in $T$ by Lemma \ref{lemma:autstructure}.\\
Since the action was arbitrary this shows that $\mathrm{Aut}(G)$ satisfies property (FA).
\end{proof}

\section{Generalisation}
In this section, we generalise the results of the previous sections to the situation of \emph{mixed} free products of finite cyclic groups where the order of the free factors might vary.
Let $m \geq 2$ and $n_1, \ldots, n_m \geq 2$ be natural numbers. Let $G$ be the free product of $m$ finite cyclic groups $\mathbb{Z}/n_1\mathbb{Z}, \ldots , \mathbb{Z}/n_m\mathbb{Z}$.

\begin{definition}
We say that the cyclic group of order $k$ \emph{occurs} $j$ times in $G$ if $\# \{i \in \{1, \ldots, m\} \mid n_i=k\} =j$.
\end{definition}

The next lemma shows that the number of occurencies is an invariant of the group $G$, i.e. it does not depend on the chosen decomposition of $G$ as a free product of finite cyclic groups.

\begin{lemma}
Let $l \geq 2$ and $r_1, \ldots, r_l \geq 2$ be natural numbers such that ${G \cong \ast_{i=1}^{l}\mathbb{Z}/r_i\mathbb{Z}}$. Then we have $l=m$ and there exists $\pi \in \mathrm{Sym}(m)$ such that $r_i=n_{\pi(i)}$ for $i=1, \ldots, m$.
\end{lemma}
\begin{proof}
By the Torsion Theorem for Free Products (Theorem 1.6, p.177 in \cite{LS77}), any element $g \in G$ of finite order is conjugate to an element of a free factor $\mathbb{Z}/n_i\mathbb{Z}$. Moreover, this free factor is uniquely determined as no non-trivial elements of different free factors are conjugate to one another. (For ${i=1, \ldots, m}$ we can define a homomorphism $G \rightarrow \mathbb{Z}/n_i\mathbb{Z}$ (as in the Proof of Lemma 2.9) whose restriction to the free factor $\mathbb{Z}/n_i\mathbb{Z}$ is injective while all other free factors lie in the kernel.)\\
As a consequence, there is only a finite set of natural numbers which appear as orders of elements in $G$, namely all the divisors of the $n_i, i=1, \ldots, m$. By the previous paragraph, this implies that there are also finitely many conjugacy classes of elements of finite order in $G$.\\
Let $k \in \mathbb{N}$. We denote by $c(k)$ the number of conjugacy classes of elements of order $k$. So $c(k)$ is an algebraic invariant of $G$.\\
A cyclic free factor $\mathbb{Z}/n_i\mathbb{Z}$ of $G$ contains an element of order $k$ if and only if $k$ divides $n_i$. Therefore the number of occurencies of the cyclic group of order $k$ in the decomposition $G = \ast_{i=1}^{m}\mathbb{Z}/n_i\mathbb{Z}$ equals $c(k)- \sum \limits_{a \in \mathbb{N}_{> 1}} c(a k)$.\\
Applying the same argument to the decomposition $G \cong \ast_{i=1}^{l}\mathbb{Z}/r_i\mathbb{Z}$, we obtain that for each $k \in \mathbb{N}$ the cyclic group of order $k$ occurs equally often in both decompositions of $G$. 
\end{proof}

In the following we give a classification depending on the number of occurencies whether the automorphism group of the free product $G$ has property FA. The next theorem is Theorem 1.2 from the introduction.

\begin{theorem}
Let $G=\ast_{i=1}^m \mathbb{Z}/n_i\mathbb{Z}$ be a free product of at least two non-trivial finite cyclic groups. If either a free factor $\mathbb{Z}/n\mathbb{Z}$ occurs exactly two or three times or if two different free factors $\mathbb{Z}/n_1\mathbb{Z}$ and $\mathbb{Z}/n_2\mathbb{Z}$ occur exactly once, then $\Aut{G}$ does not have FA.
\end{theorem}
\begin{proof}
Let $k \in \mathbb{N}, k \geq 2$ and assume $k$ occurs at least once in $G$. Let $N(k)$ be the normal subgroup generated by all free factors of order $k$. For the proof of the theorem we need the following claim.\\
\textbf{Claim}: $N(k)$ is a characteristic subgroup of $G$.\\
By Remark \ref{remchar} $iii)$, it is sufficient that $\varphi(N(k)) \subseteq N(k)$ where $\varphi \in \Aut{G}$ is any generator from the list given in Lemma \ref{lemma:gen}.\\
If $\varphi$ is a partial conjugation, it maps each free factor to a conjugate which immediately implies $\varphi(N(k)) \subseteq N(k)$.\\
If $\varphi$ is a factor automorphism it stabilises each free factor and nothing has to be proven. If $\varphi$ is a permutation it permutes free factors of order $k$ and therefore stabilises $N(k)$ as well. Hence, $N(k)$ is a characteristic subgroup of $G$.\\
Assume first that a free factor $\mathbb{Z}/n\mathbb{Z}$ occurs twice or thrice in the decomposition of $G$. For each $k \neq n$ which occurs in $G$ the subgroup $N(k)$ is characteristic in $G$. By Remark \ref{remchar} $ii)$, the subgroup $N$ generated by all $N(k)$ with $k \neq n$ is characteristic in $G$. The quotient $G/N$ is isomorphic to $\Aut{\mathbb{Z}/n\mathbb{Z}^{\ast l}}$ for $l \in \{2,3\}$.
By Lemma \ref{lemma:inducedmap}, we obtain an induced map $\Aut{G} \rightarrow \Aut{\mathbb{Z}/n\mathbb{Z}^{\ast l}}$ which is an epimorphism. By Observation \ref{obs:FA} $i)$, property FA is preserved by quotients. Since $\Aut{\mathbb{Z}/n\mathbb{Z}^{\ast l}}$ does not have property FA by Theorem 1.1, also $\Aut{G}$ does not have property FA.\\
If two free factors of orders $n_1$ and $n_2$ with $n_1 \neq n_2$ occur exactly once we proceed in analogous fashion. We can factor out the characteristic subgroup generated by all $N(k)$ for $k \notin \{n_1,n_2\}$. Thereby, we obtain an epimorphism $\Aut{G} \rightarrow \Aut{\mathbb{Z}/n_1\mathbb{Z} \ast \mathbb{Z}/n_2\mathbb{Z}}$. By Remark 3.4, $\Aut{\mathbb{Z}/n_1\mathbb{Z} \ast \mathbb{Z}/n_2\mathbb{Z}}$ does not have FA. So by the same argument as above, $\Aut{G}$ does not have FA.
\end{proof}

The following theorem is Theorem 1.3 from the introduction.

\begin{theorem}
Let $G=\ast_{i=1}^m \mathbb{Z}/n_i\mathbb{Z}$ be a free product of finite cyclic groups. If each free factor $\mathbb{Z}/n\mathbb{Z}$ occurs at least four times, then $\Aut{G}$ has FA.
\end{theorem}
\begin{proof}
We follow closely the proof of Theorem \ref{theorem:mgeq4} from the previous section:\\
Let $\mathrm{Aut}(G)$ act on a tree $T$. The first step is to show that the Fouxe-Rabinovitch subgroup $\mathrm{FR}(G)$ has a global fixed point in $T$.\\
By Corollary \ref{corollary:helly}, it is sufficent to prove that any two partial conjugations $\alpha_i^{j}$ and $\alpha_{r}^s$ have a common fixed point.\\ 
Note that each partial conjugation $\alpha_{i}^j$ has finite order $n_j$ and therefore must fix a vertex $v \in T$. If we have $s=j$ or $\{i,j\} \cap \{r,s\} = \emptyset$, then the partial conjugations $\alpha_i^j$ and $\alpha_r^s$ commute and therefore have a common fixed point in $T$. Otherwise we have again three different cases:
\begin{itemize}
\item[Case 1:] $i=r$ and $s \neq j$\\
Since $n_i$ occurs at least four times there exists $l \in \{1, \ldots m\} \setminus \{i,j,s\}$ with $n_l=n_i$. We can apply the Corollary \ref{corollary:diagonal} to the elements $a_1=\alpha_i^j, a_2=\alpha_i^s$ and $b_1=\alpha_l^j, b_2=\alpha_l^s$. The elements $a_k$ and $b_k$ commute for $k=1,2$ because they have the same operating letter ($x_j$ for $k=1$ resp. $x_s$ for $k=2$). Since $\{i,j\} \cap \{l,s\}=\emptyset$ and $\{i,s\} \cap \{j,l\}=\emptyset$ also $a_1$ commutes with $b_2$ and $b_1$ commutes with $a_2$. Thus, $a_k$ and $b_{k'}$ have a common fixed point for $k,k' \in \{1,2\}$.
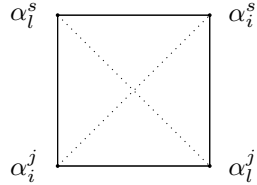
\begin{figure}[H]
\begin{center}
\begin{tikzpicture}[scale=0.5]
\draw[fill=black]  (2,2) circle (1pt);
\draw[fill=black]  (2,-2) circle (1pt);
\draw[fill=black]  (-2,2) circle (1pt);
\draw[fill=black]  (-2,-2) circle (1pt);
\draw[fill=black]  (2,2)--(-2, 2);
\draw[fill=black]  (2,2)--(2, -2);
\draw[fill=black]  (-2,-2)--(-2, 2);
\draw[fill=black]  (-2,-2)--(2, -2);
\draw[dotted]  (-2,-2)--(2, 2);
\draw[dotted]  (-2,2)--(2, -2);
\node (a1) at (2,2) [label=right:$\alpha_i^s$]{};
\node (a2) at (2,-2) [label=right:$\alpha_l^j$]{};
\node (a3) at (-2,-2) [label=left:$\alpha_i^j$]{};
\node (a4) at (-2,2) [label=left:$\alpha_l^s$]{};
\end{tikzpicture}
\end{center}
\end{figure} 
By Corollary 5.4, either $\mathrm{Fix}(a_1) \cap \mathrm{Fix}(a_2)$ or ${\mathrm{Fix}(b_1) \cap \mathrm{Fix}(b_2) }$ is non-empty. Let $\pi \in \mathrm{Aut}(G)$ be the automorphism induced by the transposition of the generators $x_i$ and $x_l$ (well-defined since $n_i=n_l$).\\
Then we have $b_k=\pi a_k \pi^{-1}$ and hence ${\mathrm{Fix}(b_k)=\pi\mathrm{Fix}(a_k)}$  for $k=1,2$. Therefore $\mathrm{Fix}(a_1) \cap \mathrm{Fix}(a_2)$  is non-empty if and only if $\mathrm{Fix}(b_1) \cap \mathrm{Fix}(b_2)$  is non-empty. Since one of those intersections must be non-empty, both of them are non-empty and $\alpha_i^j$ and $\alpha_i^s$ have a common fixed point in $T$.
\item[Case 2:] $i=s$ and $j=r$\\
Since $n_i$ and $n_j$ occur at least four times, we can choose two indices ${k,l \in \{1, \ldots, m\} \setminus \{i,j\},}$ $k \neq l$ with $n_k=n_i$ and $n_l=n_j$. Here we apply Corollary \ref{corollary:diagonal} to the elements $a_1=\alpha_i^j, a_2=\alpha_j^i$ and ${b_1=\alpha_k^l,b_2=\alpha_l^k.}$ Since $\{i,j\} \cap \{k,l\} = \emptyset$ the elements $a_1,a_2$ commute with $b_1$ and $b_2$. Thus, $a_p$ and $b_{p'}$ have a common fixed point for $p,p' \in \{1,2\}$.
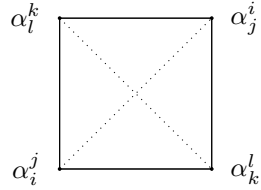
\begin{figure}[H]
\begin{center}
\begin{tikzpicture}[scale=0.5]
\draw[fill=black]  (2,2) circle (1pt);
\draw[fill=black]  (2,-2) circle (1pt);
\draw[fill=black]  (-2,2) circle (1pt);
\draw[fill=black]  (-2,-2) circle (1pt);
\draw[fill=black]  (2,2)--(-2, 2);
\draw[fill=black]  (2,2)--(2, -2);
\draw[fill=black]  (-2,-2)--(-2, 2);
\draw[fill=black]  (-2,-2)--(2, -2);
\draw[dotted]  (-2,-2)--(2, 2);
\draw[dotted]  (-2,2)--(2, -2);
\node (a1) at (2,2) [label=right:$\alpha_j^i$]{};
\node (a2) at (2,-2) [label=right:$\alpha_k^l$]{};
\node (a3) at (-2,-2) [label=left:$\alpha_i^j$]{};
\node (a4) at (-2,2) [label=left:$\alpha_l^k$]{};
\end{tikzpicture}
\end{center}
\end{figure} 
By Corollary 5.4, either $\mathrm{Fix}(a_1) \cap \mathrm{Fix}(a_2) $ or $\mathrm{Fix}(b_1) \cap \mathrm{Fix}(b_2) $ is non-empty. We conclude with the same argument as in case 1 by taking $\pi \in \mathrm{Aut}(G)$ to be the automorphism induced by interchanging $x_i$ with $x_k$ and $x_j$ with $x_l$.
\item[Case 3:] $i \neq s$ and $j=r$\\
As $n_j$ occurs at least four times there exists $l \in \{1, \ldots, m\} \setminus \{i,j,s\}$ with $n_l=n_j$. We use Corollary \ref{corollary:diagonal} in the situation $a_1=\alpha_i^j, a_2=\alpha_j^s$ and $b_1=\alpha_l^i, b_2=\alpha_s^l$. The elements $a_1, b_2$ resp. $a_2,b_1$ have a common fixed point since they commute. Moreover, $a_1$ and $b_1$ have a common fixed point by case 1. The elements $a_2$ and $b_2$ have the same operating letter, hence also have a common fixed point.
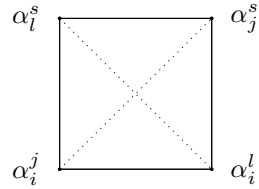
\begin{figure}[H]
\begin{center}
\begin{tikzpicture}[scale=0.5]
\draw[fill=black]  (2,2) circle (1pt);
\draw[fill=black]  (2,-2) circle (1pt);
\draw[fill=black]  (-2,2) circle (1pt);
\draw[fill=black]  (-2,-2) circle (1pt);
\draw[fill=black]  (2,2)--(-2, 2);
\draw[fill=black]  (2,2)--(2, -2);
\draw[fill=black]  (-2,-2)--(-2, 2);
\draw[fill=black]  (-2,-2)--(2, -2);
\draw[dotted]  (-2,-2)--(2, 2);
\draw[dotted]  (-2,2)--(2, -2);
\node (a1) at (2,2) [label=right:$\alpha_j^s$]{};
\node (a2) at (2,-2) [label=right:$\alpha_i^l$]{};
\node (a3) at (-2,-2) [label=left:$\alpha_i^j$]{};
\node (a4) at (-2,2) [label=left:$\alpha_l^s$]{};
\end{tikzpicture}
\end{center}
\end{figure}
We argue again as in case 1 taking $\pi \in \mathrm{Aut}(G)$ to be the automorphism induced by the transposition of the generators $x_j$ and $x_l$.
\end{itemize}
So any two partial conjugations $\alpha_i^{j}$ and $\alpha_{r}^s$ have a common fixed point which implies that $\mathrm{FR}(G)$ fixes a vertex in $T$.\\ 
Since $\mathrm{FR}(G)$ is normalised by factor automorphisms and permutations by Lemma \ref{lemma:FRnormal}, the subgroup $F \rtimes \Sym{m} \subseteq \Aut{G}$  stabilises the fixed point tree $T^{\mathrm{FR}(G)}$. Since $F \rtimes \Sym{m}$ is finite, it fixes a vertex $v$ in $T^{\mathrm{FR}(G)}$. Then, $v$ is a global fixed point of $\mathrm{Aut}(G)$ in $T$ by Lemma \ref{lemma:autstructure}.\\
Since the action was arbitrary this shows that $\mathrm{Aut}(G)$ satisfies property (FA). 
\end{proof}

\end{document}